\documentclass[a4paper,draft]{amsart}
\usepackage{amssymb}

\usepackage{mathrsfs}  

\newtheorem{thm}{Theorem}[section]
\theoremstyle{plain}
\newtheorem*{thm*}{Theorem}
\newtheorem{prop}[thm]{Proposition}
\newtheorem*{prop*}{Proposition}
\newtheorem{lem}[thm]{Lemma}

\theoremstyle{definition}
\newtheorem{dfn}[thm]{Definition}

\newtheorem*{ex*}{Example}
\theoremstyle{definition}

\DeclareMathOperator{\iind}{Ind}

\DeclareMathOperator{\Span}{span}
\DeclareMathOperator{\supp}{supp}
\DeclareMathOperator{\SG}{S}

\newcommand{\ca}[1]{\mathcal{#1}}
\newcommand{\fH}[1]{\mathcal{H}_{#1}}

\newcommand{\ind}[3]{\iind_{#1}^{#2}{#3}}

\newcommand{\s}[0]{\Sigma}
\newcommand{\gus}[0]{G^{(0)}}
\newcommand{\Gus}[1]{{#1}^{(0)}}
\newcommand{\gudu}[0]{G^{u}_{u}}
\newcommand{\gud}[2]{G^{#1}_{#2}}

\newcommand{\gt}[0]{\SG}
\newcommand{\gbg}{\s \ast \gt}
\newcommand{\bhb}{\gus\ast \fH{}}

\renewcommand{\leq}{\leqslant}

\title{THE ORBIT SPACE OF GROUPOIDS WHOSE $C^*$-ALGEBRAS ARE GCR}

\author[DW van Wyk]{\bfseries Daniel W van Wyk} 

\address{  Department of Mathematics and Statistics \\ 
 University of Otago   \\ 
Dunedin, \newline
New Zealand 
}
\email{dwvanwyk@maths.otago.ac.nz}

\begin{document}

\vspace{18mm} \setcounter{page}{1} \thispagestyle{empty}

\begin{abstract}
Let $G$ be second countable locally compact Hausdorff groupoid
with a continuous Haar system. We remove the assumption of amenability in a theorem by 
Clark about GCR groupoid $C^*$-algebras. We show that if the groupoid $C^*$-algebra of $G$ is GCR then the orbits of $G$ are locally closed.
\end{abstract}

\maketitle

\section{INTRODUCTION}
$C^*$-algebras can be divided into classes based on their representation
theory. Two such classes of well-behaved $C^*$-algebras are GCR and CCR $C^*$-algebras. 
Let $\ca{K}(\fH{})$ denote the compact operators on a Hilbert space $\fH{}$.
A $C^*$-algebra $\ca{A}$ is called CCR, if for every irreducible representation
$\pi:\ca{A}\to B(\fH{\pi})$ we have $\pi(\ca{A})=\ca{K}(\fH{\pi})$.
 $\ca{A}$ is called GCR if for every irreducible representation
$\pi$ we have $\pi(\ca{A})\supset\ca{K}(\fH{\pi})$.

We investigate the orbit spaces of groupoids whose $C^*$-algebras are GCR or CCR. The techniques used in the GCR and CCR cases are quite different. Therefore, in this paper, the first of two, we treat the classes of groupoids whose $C^*$-algebras are GCR, or equivalently type I. 
In \cite{Cla07} Clark gives the following characterization for 
groupoids whose $C^*$-algebras are GCR: 
\begin{thm*}[Clark]
Let $G$ be a second-countable locally compact Hausdorff \\ groupoid with a Haar system. 
Suppose that all the stability subgroups of $G$ are amenable. Then $C^*(G)$ is 
GCR if and only if the orbit space is $T_{0}$ and the stability 
subgroups of $G$ are GCR. 
\end{thm*}
 
Clark's theorem generalizes a theorem for $C^*$-algebras of transformation groups by Gootman, \cite{Goo73}. However,
Gootman does not assume that the stabilizers are amenable. Due to the lack of 
an amenability assumption in Gootman's GCR characterization, Clark 
conjectures that the amenability hypothesis in the groupoid characterization 
is unnecessary \cite{Cla07}. We provide an an affirmative answer to Clark's conjecture.

Clark defines a map from the orbit space of the groupoid into the spectrum
$C^*(G)^{\wedge}$ of the groupoid $C^*$-algebras $C^*(G)$, and requires amenable stabilizers to show that this 
map is continuous. Clark's GCR proof only uses the continuity of this map, and thus amenability,  to prove that ``if $C^*(G)$ is GCR then the orbit space is $T_{0}$."  Therefore, to remove 
the amenability assumption from Clark's GCR characterization we only need to show that if $C^*(G)$ is GCR, then the orbit space is $T_{0}$. 

Clark uses a different approach to Gootman. We  show that Gootman's approach can be adapted to the groupoid setting. In a second countable locally compact Huasdorff groupoid the stabilizers always vary measurably, even when the stabilizers don't vary continuously, \cite[Lemma 1.6]{Ren91}. With an appropriate measure we construct a direct integral representation of $C^*(G)$ from representations that are induced from stabilizers. We prove a groupoid version of Lemma 4.2 in \cite{Eff65} by Effros, that imposes a condition on the measure which ensures that the direct integral is a type I representation. Then we prove the contrapositive: if the orbit space is not $T_0$, then Ramsay's Mackey-Glimm dichotomy for groupoids (\cite[Theorem 2.1]{Ram90}) ensures we have a non-trivial ergodic measure on the unit space. Then by our groupoid version of Effros' lemma, if the measure is non-trivially ergodic, then the direct integral representation cannot be type I. Since a $C^*$-algebra is GCR if and only if it is type I (\cite{Kap51, Gli61, Sak66}), the result follows.


\section{PRELIMINARIES} \label{ch:Prelim}
Throughout $G$ is a second-countable locally compact Hausdorff groupoid, with 
a continuous Haar system $\{\lambda^u\}_{u\in\gus}$ (see \cite{Ren80} for these definitions). Let $r$ and $s$ denote the range and source maps, respectively, from $G$ onto the unit space $\gus$. For $u\in\gus$ we let $G^u:=r^{-1}(u)$, $G_u:=s^{-1}(u)$ and the stabilizer or stability subgroup at $u$ is $\gudu:=r^{-1}(u)\cap s^{-1}(u)$. 
For $x\in G$, the map $R(x):=(r(x),s(x))$ defines an 
equivalence relation $\sim$ on $\gus$. For $u\in\gus$ we let 
$[u]:=\{v\in\gus :u\sim v\}$ denote the orbit of $u$. The orbit space $\gus / G$
is the quotient space for the equivalence relation $\sim$ on $\gus$. 

Throughout $C_{c}(X)$ denotes the continuous compactly 
supported functions from the topological space $X$ into $\mathbb{C}$.
	If $f,g\in C_{c}(G)$, then
	$$f\ast g (x):=\int_{G}{f(y)g(y^{-1}x) \: d\lambda^{r(x)}(y)}$$
	and
	$$ f^{*}(x):= \overline{f(x^{-1})} ,$$
	define convolution and involution operations on $C_{c}(G)$, respectively.  
	With these operations $C_{c}(G)$ is a *-algebra. Let $\fH{}$ be a separable Hilbert space and $B(\fH{})$ the bounded linear operators on $\fH{}$.  A \textit{representation} of $C_{c}(G)$ is a *-homomorphism $\pi:C_{c}(G)\to B(\fH{})$	such that $||\pi(f)||\leq ||f||_I$, where $||f||_I$ is the I-norm on $C_{c}(G)$ (see \cite{Ren80} for the I-norm). Then $C^{*}(G)$ is the completion of $C_{c}(G)$ in the norm
$$||f||:=\{\sup||\pi(f)||:\pi \text{ is a representation of }C_{c}(G)\}.$$	
We assume all representations are non-degenerate.

We use representations of $C^*(G)$ which are induced from the trivial representations of stabilizers. 
Fix $u\in \gus$, $f\in C_{c}(G)$ and $\phi,\psi\in C_{c}(G_u)$. The trivial representation $1_u$ of $\gudu$ is given by $1_u(t)=1$ for every $t\in \gudu$. 
The integrated form of $1_u$ is the representation $\pi_{1_u}:C_{c}(\gudu)\to \mathbb{C}$ given by 
\begin{equation*}
\pi_{1_u}(a):= \int_{\gudu}{a(t)\:\Delta_{u}(t)^{-1/2} \: d\beta^u(t)}, 
\end{equation*}
which extends to give a representation of $C^*(\gudu)$.  Note: the modular function in the integrand above is due to the fact that we view the stabilizers as \textit{subgroupoids}.
 Clark (\cite{Cla07}) and Ionescu and Williams (\cite{IonWil09}) show that we can induce $\pi_{1_u}$ to get a representation 
$$\ind{\gudu}{G}{\pi_{1_u}}:C^*(G)\to B(\fH{u}). $$
We briefly describe how the induced representation is constructed and introduce some notation.
Clark and Ionescu and Williams use induction via Hilbert modules. They show that 
$C_c(G_u)$ is a right $C^*(\gudu)$-pre-Hilbert module,  where the right inner product on $C_c(G_u)$ is given by  
\begin{eqnarray}\label{eq:ModuleInnerProd} 	
 \langle \psi, \phi \rangle_{C_{c}(\gudu)}(t)= \psi^{*}\ast\phi (t).
\end{eqnarray}
The completion of $C_{c}(G_u)$ in the inner product (\ref{eq:ModuleInnerProd}) gives a right $C^*(\gudu)$-Hilbert module $\mathcal{X}$.
Furthermore, for all $\gamma \in G_u$,
\begin{eqnarray*}\label{eq:AdjointActionOnModule}
f\cdot \phi (\gamma)&:=&\int_{G}{f(\eta)\phi(\eta^{-1}\gamma)\: d\lambda^{r(\gamma)}(\eta)} \\
					&=& f\ast \phi (\gamma)
\end{eqnarray*}
defines an action of $C_c(G)$ on $C_c(G_u)$ as adjointable operators, which extends to an 
action of $C^*(G)$ on $\mathcal{X}$. 
Before the induced representation, we first consider the representation space. 
Define an inner product on $C_c(G_u)$ by 
\begin{eqnarray} 
(\phi\mid \psi)_{u}&=& \pi_{1_u}(\langle \psi,\phi\rangle_{_{C^*(\gudu)}}) \nonumber\\
				&=& \pi_{1_u}(\psi^{*}\ast \phi ) \nonumber \\
				&=&\int_{\gudu}{\psi^{*}*\phi(t)\Delta_{u}(t)^{-\frac{1}{2}}
			\: d\beta^{u}(t)}.  \label{eq:InducedInnerprod}
\end{eqnarray}
Denote the completion of  $C_c(G_u)$ in the inner product in (\ref{eq:InducedInnerprod}) by $\fH{u}$. Then the induced representation  $\ind{\gudu}{G}{\pi_{1_u}}:C_c(G)\to C_c(G_u)$ is defined by
\begin{equation} \label{TrivInducedAction}
\ind{\gudu}{G}{\pi_{1_u}}(f)(\phi)=f\ast \phi.
\end{equation} 
By \cite[Proposition 2.66]{MoritaEq}, $ \ind{\gudu}{G}{\pi_{1_u}}$ extends to 
give a representation of $C^*(G)$ as bounded linear 
operators on the Hilbert space $\fH{u}$.
To simply notation we write
$$l^u:=\ind{\gudu}{G}{\pi_{1_u}}, $$
for any $u\in \gus$. Finally, each representation $l^u, u\in\gus$, is an irreducible representation of $C^*(G)$, \cite{Cla07a, IonWil09}.


\section*{\textbf{GCR GROUPOID $C^*$-ALGEBRAS}} \label{ch:GCR}
Section \ref{measuresongroupoids} addresses the measurability of a map on the units space which is used to construct a direct integral representation. In Section  \ref{directint} we construct a Borel Hilbert bundle and 
a direct integral representation. The direct integral representation acts on the Hilbert space of square integrable sections of the Borel Hilbert bundle. We also prove Proposition \ref{factorrep}, a groupoid version of a result by Effros, giving a condition on the measure used for the direct integral representation to be type I. Section \ref{main} contains our main GCR result, Theorem \ref{mainGCRthm}. 

\section{A BOREL MAP ON THE UNIT SPACE} \label{measuresongroupoids}

For the construction of the direct integral representation, we need Proposition 
\ref{Borelprop} below, which gives the measurability of certain maps on the unit space.
These maps have the form of an integral, where 
the group with respect to which we integrate and its Haar measure depend on 
the particular unit in $\gus$. 
This this dependence is fine, since Renault shows \cite[Lemma 1.5]{Ren91} that 
the stabilizer map $u\mapsto \gudu$ from $\gus$ to 
the space of all closed subgroups with the Fell topology,
always varies measurably. However, the integrand also has
a modular function that depends on the unit in $\gus$. 
We show that despite the modular function 
these maps are still Borel. Specifically we show:

\begin{prop} \label{Borelprop}
		Let $u\in\gus$ and $f\in C_{c}(G)$. There exists a Haar measure $\beta^u$ on $\gudu$ 
		with associated modular function $\Delta_{u}$ such that	the map 
		$$ u\mapsto \int_{\gudu}{f(t)\:\Delta_{u}(t)^{-\frac{1}{2}} \: d\beta^u (t)}$$
		is Borel.
	\end{prop}


For a locally compact Hausdorff space $X$
let $\mathscr{B}(X)$ be the Borel $\sigma$-algebra on $X$.  
A \textit{Borel measure} is a positive Radon measure on $\mathscr{B}(X)$. 

Let 
$$\gt:=\bigcup_{u\in \gus}\gud{u}{u}$$
denote the stability subgroupoid of $G$.
Then the  range and source maps agree on $\gt$ and $\Gus{\gt}=\gus$. Because $\gus$ is 
Hausdorff, $\gt$ is a closed subset of $G$.

Let $\mathscr{C}(\gt)$ be the set of all closed subsets of $\gt$ with the Fell topology.
Then $\mathscr{C}(\gt)$ is a compact Hausdorff space, \cite[Proposition H.3]{CrossedProd}.
Since $\gt$ is second countable, so is $\mathscr{C}(\gt)$.  Let 
$$\Sigma:=\{H\in\mathscr{C}(\gt):H \text{ \textit{is a closed subgroup of} }\gt \}.$$	
Give
$$\gbg:=\{(H,\gamma)\in \s \times \gt: H\in\s, \gamma\in H\}$$
the relative topology inherited from the product topology on 
$\s\times\gt$. Note: $\gbg$ is a group 
bundle groupoid, and its unit space is identified with $\s$. 
We show that $\gbg$ is locally compact Hausdorff.

\begin{lem}
The groupoid $\gbg$ is second-countable, locally compact and Hausdorff.
\end{lem}	
\begin{proof}
We first show that $\s\times\gt$ is second-countable,
locally compact and Hausdorff. Then we show that $\gbg$ is closed in $\s\times\gt$.

As a subspace $\gt$ is automatically second countable and Hausdorff. 
Since $\gt$ is closed in $G$, it locally compact. As a
subspace of $\mathscr{C}(\gt)$, $\s$ 
is second-countable and Hausdorff. 
Since $\s\cup \{\emptyset\}$ is compact in $\mathscr{C}(\gt)$ and $\mathscr{C}(\gt)$ is Hausdorff,
we have that $\s\cup \{\emptyset\}$ is closed in $\mathscr{C}(\gt)$.
Also, $\mathscr{C}(\gt)$ Hausdorff implies 
$\{\emptyset\}$ is closed in $\mathscr{C}(\gt)$.		
Thus, $\s=(\s\cup \{\emptyset\})\backslash \{\emptyset\}$ is open in 
$\mathscr{C}(\gt)$, and hence locally compact.
Since both $\gt$ and $\s$ are second-countable locally compact and Hausdorff,
so is $\s\times\gt$. 
			
We show that 
$\gbg$ is closed in $\s\times\gt$. Suppose that 
$\{(H_{i},\gamma_{i})\}$ is a sequence in 
$\gbg$ converging to some $(H,\gamma)$ in $\s\times\gt$. 
Then $H_{i} \to H$ in $\mathscr{C}(\gt)$,
$\gamma_{i}\in H_{i}$	for every $i$, and 
$\gamma_{i}\to \gamma$. Thus $\gamma\in H$, by the characterization
of convergence in $\mathscr{C}(\gt)$ (Lemma H.2, \cite{CrossedProd}). Then
$(H,\gamma)\in \gbg$, which shows that $\gbg$ is closed in $\s\times\gt$. 
Since $\gbg$ is closed	in $\s\times\gt$, it is	locally compact. Second-countability 
and Hausdorffness are automatic for subspaces. Thus 
$\gbg$ is second-countable, locally compact and Hausdorff.
\end{proof}
		
The next lemma is used in the proof of 	Proposition \ref{Borelprop}. It 
shows that we can associate with every $f\in C_{c}(G)$ a function in 
$C_{c}(\gbg)$.
		
\begin{lem} \label{compsuppfunct}
	Suppose that $f\in C_{c}(G)$. For all $(H,\gamma)\in \gbg$ define 
	$$F(H,\gamma):=f(\gamma).$$ 
	Then $F\in C_{c}(\gbg)$.
\end{lem}
	
\begin{proof}
Note that $F$ is just $f$ composed with the projection onto 
the second coordinate. Since both $f$ and the projection is continuous, it follows that 
$F$ is also continuous. 

We show that $F$ has compact support. Let $\{(H_{i},\gamma_{i})\}$ 
be a sequence in the support of
$F$. Then $\{\gamma_{i}\}$ is a sequence in the support of $f$. 
Since $f$ has compact support, $\{\gamma_{i}\}$
has a convergent subsequence such that (after relabelling) $\gamma_{j} \to \gamma$ in
$\supp(f)$.	Then $\{H_{j}\}$ is a sequence in the compact 
space $\s\cup\{\emptyset\}$. Thus $\{H_{j}\}$ has a convergent subsequence, 
such that (after relabelling) $H_{k}\to H$ in $\s\cup\{\emptyset\}$. 
Since $\gamma_{k} \in H_{k}$ and $\gamma_{k} \to \gamma$, 
the characterization of convergent sequences in $\mathscr{C}(\gt)$ (Lemma H.2, \cite{CrossedProd}) implies that $\gamma\in H$.
Thus $(H_{i},\gamma_{i})$ has a convergent subsequence $(H_{k},\gamma_{k})$ 
converging to $(H,\gamma)$ in $\supp(F)$.
Thus $\supp(F)$ is compact., and so $F\in C_{c}(\gbg)$.
\end{proof}

In \cite{Ren91}, Renault deals with possibly non-Hausdorff groupoids. He therefore
introduces locally conditionally compact groupoids. 
We use a result from \cite{Ren91} in Proposition \ref{Borelprop} to claim 
that $\gbg$ has a Haar system. So we need to know 
that $\gt=\cup_{u\in \gus}\gud{u}{u}$ is locally conditionally compact.  

 	A set $L$ in a groupoid $G$ is called \textit{left (respectively right) 
 	conditionally compact} 
	if for every compact set $K\subset \gus$, the set $KL=L\cap r^{-1}(K)$ 
	(respectively $LK=L\cap s^{-1}(K)$) is compact. The set $L$ is called
	\textit{conditionally compact} if it is both left and right conditionally
	compact. If every point in the groupoid has a conditionally compact neighbourhood, then $G$ is \textit{locally conditionally compact}.

\begin{lem} \label{lcc}
	Let $G$ be a locally compact Hausdorff groupoid.
	Then \newline $\gt=\cup_{u\in \gus}\gud{u}{u}$ is locally conditionally 
	compact.  
\end{lem}
		
	\begin{proof}
		Since $\gt$ is closed in $G$, it is locally compact and Hausdorff in the 
		relative topology from $G$. 
		Let $\gamma \in \gt$. Then $\gamma$ has a compact neighbourhood
		$L\subset\gt$. Let $K$ be any compact set in $\gus$. 
		Since $\gus$ is Hausdorff, $K$ is 
		closed in $\gus$. Since $r_{\gt}:=r|_{\gt}$ is continuous, 
		$r_{\gt}^{-1}(K)$ is closed in $\gt$. 
		The intersection of the compact set $L$ and the closed set $r_{\gt}^{-1}(K)$
		is compact in $\gt$. Hence $\gt$ is left conditionally
		compact. Put $s_{\gt}:=s|_{\gt}$. 
		Since $r_{\gt}=s_{\gt}$ on $\gt$, it follows that $L\cap s^{-1}(K)$ is 
		also compact. Thus $\gt$ is right conditionally
		compact. Since every $\gamma\in G$ has a conditionally compact neighbourhood,  
		$\gt$ is locally conditionally compact. 
	\end{proof}

Now we prove Proposition \ref{Borelprop}.

\begin{proof}[Proof of Proposition \ref{Borelprop}]	\label{proofContHaarChoice}
Recall that $\gbg$ is a group-bundle groupoid with unit space $\s$.
By Lemma \ref{lcc}, $\gt$ is locally conditionally compact. Hence, by 
\cite[Corollary 1.4]{Ren91}, the groupoid $\gbg$ has a continuous  Haar system 
$\{\beta^H\}_{H\in\s}$.
			
Set $\beta^u:=\beta^{\gudu}$ for each $u \in \gus$. We claim that each 
$\beta^u$ is a Haar measure on $\gudu$. From the definition of a Haar 
system, each $\beta^{u}$ is a non-zero Radon measure such that 
$\supp(\beta^{u})=\supp(\beta^{\gudu})=\gudu$. 
With $f\in C_{c}(G)$, we set $F(H,\gamma):=f(\gamma)$. Then $F\in C_{c}(\gbg)$ 
by Lemma \ref{compsuppfunct}. 
 Suppose that $x,\gamma\in \gudu$. Then the left invariance of  the
 Haar system $\{\beta^H\}_{H\in\s}$ gives		 
\begin{eqnarray*}
  \int_{\gudu}{f(x\gamma)d\beta^{u}(\gamma)}&=& 
 			\int_{\gud{u}{u}}{F(\gudu,x\gamma)\: d\beta^{\gudu}(\gamma) } \\
	  &=& \int_{\gud{u}{u}}{F[(\gudu,x)(\gudu,\gamma)]\: d\beta^{\gudu}(\gamma)}  \\
	  &=& \int_{\gud{u}{u}}{F(\gudu,\gamma)\: d\beta^{\gudu}(\gamma) }  \\
	   &=& \int_{\gud{u}{u}}{f(\gamma)\: d\beta^{u}(\gamma) }.
\end{eqnarray*}
Hence every $\beta^{u}$ is a non-zero left invariant Radon measure on $\gudu$, that is,
$\beta^{u}$ is a Haar measure on $\gudu$.

For every $H\in\s$, let $\Delta_{H}$ denote the modular function of the 
group $H$ corresponding to a Haar measure $\beta^{H}$. For the particular 
case where $H=\gudu$ for some $u\in \gus$, we write $\Delta_{u}:=\Delta_{\gudu}$.
By \cite[Lemma 5.3]{Cla07} the map $D:\gbg \to \mathbb{R}$
given by $D(H,\gamma)=\Delta_{H}(\gamma)$ is continuous. 
Hence the pointwise product $F\cdot D^{-1/2}$ belongs to $C_{c}(\gbg)$. Since 
$\{\beta^H\}_{H\in\s}$ is a Haar system, the map 
\begin{equation}\label{map:ctsUnitspace}
\gudu\mapsto \int_{\gudu}{F(\gudu,\gamma)D(\gudu,\gamma)^{-1/2}\: d\beta^{\gudu}(\gamma)}
\end{equation}
is continuous. Also, by \cite[Lemma 1.5]{Ren91} the stabilizer  map $u\mapsto \gudu$ is Borel. 
By composing the stabilizer map with (\ref{map:ctsUnitspace}), we get that 
\begin{eqnarray*}
				u\mapsto \int_{\gudu}{F(\gudu,\gamma)D(\gudu,\gamma)^{-1/2}\: d\beta^{\gudu}(\gamma)}
				=\int_{\gudu}{f(\gamma)\Delta_{u}(\gamma)^{-\frac{1}{2}}\: d\beta^u (\gamma)}
\end{eqnarray*}
is Borel.
\end{proof}

\section{A DIRECT INTEGRAL REPRESENTATION OF $C^*(G)$}  \label{directint}
In this section we construct a direct integral representation of 
$C^*(G)$, and give a condition on the measure which ensures that the direct integral representation is type I. The idea is to associate with every $u\in \gus$
the irreducible representation $l^u$ of $C^*(G)$. Then with an appropriate 
measure on $\gus$, we combine all the $l^u$'s and their representations
spaces in a measurable way to form a new representation
of $C^*(G)$. We construct our direct-integral Hilbert space via a Borel Hilbert bundle, as
defined Appendix F.2 of \cite{CrossedProd}. Thus
our first goal is to construct a Borel Hilbert bundle.

A \textit{Polish space} is a topological space which is homeomorphic to 
a separable complete metric space. A subset $E$ in a Polish space $X$ 
is  \textit{analytic} if there is a Polish space 
$Y$ and a continuous map $f:Y \to X$ such that $f(Y)=E$.

Suppose that $\{\fH{x}\}_{x\in X}$ is a family of non-zero Hilbert spaces indexed by 
	a set $X$. Let 
	$$X\ast \fH{}:=\{(x,h):x\in X, h\in \fH{x}\}$$
	be the disjoint union and
	$\rho:X\ast \fH{} \to X$  the projection onto the first 
	coordinate. A \textit{section} is a function $f:X \to X\ast \fH{}$ 
	such that $\rho\circ f(x)=x$. So a section has the form 
	$f(x)=(x, \hat{f}(x))$, with $ \hat{f}(x) \in \fH{x}$. As is common in the literature, we don't always make a distinction between $f$ and $\hat{f}$.

We recall the definition of a Borel Hilbert bundle. 
	
\begin{dfn}\cite[Definition F.1]{CrossedProd}\label{dfnbhb}
	Let $\fH{}=\{\fH{x}\}_{x\in X}$ be a family of separable 
	Hilbert spaces indexed by an analytic 
	Borel space $X$. Then $(X\ast \fH{},\rho)$ is a \textit{Borel Hilbert bundle}
	if $X\ast \fH{}$ has a Borel structure such that:
	\begin{enumerate}
	\itemsep0em
		\item[(a)] $\rho$ is a Borel map  
		\item[(b)] there is a sequence $\{f_{n}\}$ of sections  such that 
		\begin{enumerate}
			\item[(b1)] the maps $\tilde{f_{n}}:X\ast \fH{} \to \mathbb{C}$, defined by 
				$$\tilde{f_{n}}(x,h):=( \hat{f}_{n}(x)\mid h)_{\fH{x}},$$
				are Borel for each $n$, 
			\item[(b2)] for every $m$ and $n$,
				$$ x \mapsto ( \hat{f}_{n}(x)\mid \hat{f}_{m}(x))_{\fH{x}}$$
				is Borel, and 
			\item[(b3)] the functions $\{\tilde{f_{n}}\}\cup\{\rho\}$ 
				separate points of $X\ast \fH{}$. \\
		\end{enumerate}
	\end{enumerate}
The sequence $\{f_{n}\}$ is called a \textit{fundamental sequence} for $(X\ast \fH,\rho)$. 
A \textit{Borel section} is a section $f$ of $(X\ast \fH,\rho)$ such that
	\begin{equation*} 
		x \mapsto ( \hat{f}(x)\mid \hat{f}_{n}(x))_{\fH{x}} 
	\end{equation*}
	is Borel for all $n$. Let $\ca{B}(X\ast \fH{})$ be the set of all Borel sections.
\end{dfn}

Suppose that $u\in\gus$. Recall that $l^u$ denotes the irreducible representation of $C^*(G)$ that acts on the completion $\fH{u}$ of $C_{c}(\gud{}{u})$.
We show that $$\bhb:=\{(u,h):u\in \gus, h\in\fH{u}\}$$
	 is a Borel Hilbert bundle by invoking the following theorem: 
	
\begin{prop}\cite[Proposition F.8]{CrossedProd} \label{bhbthm} 
Suppose that $X$ is an analytic Borel space and that $\fH{}=\{\fH{x}\}_{x\in X}$ 
is a family of separable Hilbert spaces. Suppose that $\{f_{n}\}$ 
is a countable family of sections of $X\ast \fH{}$
such that conditions $(b2)$ and $(b3)$ of Definition \ref{dfnbhb} are satisfied. 
Then there is a unique analytic Borel structure on 
$X\ast \fH{}$ such that $(X\ast \fH,\rho)$
becomes an analytic Borel Hilbert bundle and $\{f_{n}\}$ is a fundamental sequence.
 \end{prop}

To apply Proposition \ref{bhbthm} we need candidates for sections that satisfy
conditions $(b2)$ and $(b3)$ of Definition \ref{dfnbhb}. We use a sequence of functions from $C_{c}(G)$ that is dense in the inductive limit topology, and such that their restrictions to $C_{c}(G_u)$ are dense in $\fH{u}$, for every $u\in\gus$. 
It is almost certainly well-known to experts that $C_{c}(G)$ is separable in the inductive limit topology. We still give a proof as the construction is used to show that the sequence is also dense in each $\fH{u}$ when restricted to $C_{c}(G_u)$. 
When considering the inductive limit 
topology, it will suffice to know that if $f_{i}\to f$ 
uniformly and the $\supp (f_{i})$ is eventually contained in
some compact set, then 
$f_{i}\to f$ in the inductive limit topology.
Note: the converse of this statement is false (see 
\cite[Example D.9]{MoritaEq} for a counter example).

\begin{lem} \label{denseinductivelim} 
	There is a countable sequence of functions $\{f_{i}\}$ in 
	$C_{c}(G)$ which is dense in $C_{c}(G)$ in the inductive limit topology.
	Moreover,	the restrictions $\{f_{i}|_{G_u}\}$ are dense in $\fH{u}$
	for every $u \in \gus$.
\end{lem}
		
	\begin{proof}
		Suppose that $U$ is an open set with compact closure in $G$. 
		Then every $f\in C_{0}(U)$  
		extends to $C_{c}(G)$ by putting $f(x)=0$ if $x\notin U$. 
		In this way we view $C_{0}(U)$ as a *-subalgebra of $C_{c}(G)$ 
		consisting of functions which vanish outside of the 
		compact set $\overline{U}$. Since $G$ is second-countable and locally compact,
		we can write $G$ as the union of a sequence of open sets $\{U_{i}\}$  such 
		that $U_{i}\subset U_{i+1}$ and  
		$\overline{U_{i}}$ is compact. The set $U_{i}$ is second-countable for 
		every $i\in \mathbb{N}$.
		Thus $C_{0}(U_{i})$ is a separable Banach space 
		in the uniform norm $||\cdot||_{\infty}$.  
		For each $i\in \mathbb{N}$, let $\{f_{i_{k}}\}_{k}$  be a countable 
		dense set in $C_{0}(U_{i})$, 
		which we view as a subset of $C_{c}(G)$. 
			
		Let $f\in C_{c}(G)$. We show that $f$ can be approximated by functions
		of the form  $\{f_{i_{k}}\}_{k}$ in the inductive limit topology.
		Since $\{U_{i}\}$ is an increasing sequence of open sets 
		with compact closure,  it follows that 
		$\supp(f)\subset U_{i}$ for some $i$. Then $f\in C_{0}(U_{i})$ and there is
		a subsequence $\{f_{i_{k(j)}}\}_{j}$ of $\{f_{i_{k}}\}_{k}$ such that
		$f_{i_{k(j)}} \to f$
		uniformly  in $C_{0}(U_{i})$ as $j\to\infty$.
		Viewing each $f_{i_{k(j)}}$ as an element of $C_{c}(G)$ we have 
		$\supp(f_{i_{k(j)}}) \subset \overline{U_{i}}$ for all $j$. 
		Thus $f_{i_{k(j)}}$ converges to $f$ in $C_{c}(G)$ 
		in the inductive limit topology. Since $f$ was arbitrary, it follows that 
		$\{f_{i_{k}}\}_{i,k}$ is countable and dense in $C_{c}(G)$ in the 
		inductive limit topology.
			
		Fix $u\in\gus$. We show that $\{f_{i_{k}}\}_{i,k}$ restricted
		to $G_u$ is dense in $\fH{u}$.
		Suppose that $h\in \fH{u}$ and let $||\cdot||_{u}$ denote the norm
		defined by the 
		inner product on $\fH{u}$. Let $\epsilon >0$. Then there is 
		an $f\in C_{c}(\gud{}{u})$
		such that 
		\begin{equation} \label{estimate1}
		||h-f||_{u}<\frac{\epsilon}{2}.
		\end{equation}
		Since
		$\gud{}{u}$ is closed in $G$, the support of $f$ is compact in $G$. By 
		\cite[Lemma 1.42]{CrossedProd}, we extend $f$ to $C_{c}(G)$ (using
		the same notation $f$ for the extension).
		Since $\{U_{i}\}$ is an increasing chain of relatively compact sets,
		there is an $i$ 
		such that $\supp(f)\subset U_{i}$. Then (after relabelling) there is a
		subsequence
		$\{f_{j}\}$ of $\{f_{i_{k}}\}_{i,k}$ such that $f_{j}\to f$ uniformly and  
		$\supp(f_{j})\subset\overline{U}_{i}$ for every $j$. 
		
We now consider two cases. First suppose that 
$\overline{U}_{i}\cap\gud{u}{u}=\emptyset$. Then  \newline
$\supp(f-f_j)\cap\gudu=\emptyset$, and thus  
\begin{eqnarray}
||f-f_{j}||_{u}&=& \left[ (f-f_{j},f-f_{j})_u  \right]^{1/2} \nonumber \\
				&=&  \left[ \int_{\gudu}{\left((f-f_j)^{*}\ast(f-f_j)\right)(t)
		  		 \: \Delta_{u}(t)^{-1/2}\:d\beta^u(t)}  \right]^{1/2}  \nonumber\\
		  		 &=& \left[\int_{\gudu}
			{\int_{\gudu}{(f-f_j)^{*}(s)
			(f-f_j)(s^{-1}t)\:d\beta^u(s)} \:\Delta_{u}(t)^{-1/2}
			\:d\beta^u(t)}\right]^{1/2}  \nonumber \\
			&=& 0. \label{eq:Case1Dense}
\end{eqnarray}

Second, suppose that
$\overline{U}_{i}\cap\gud{u}{u}\neq \emptyset$. 	Put 
$$M^{2}:=\left(\sup\{\Delta_{u}^{-1/2}(t):t\in 
\overline{U}_{i}\cap\gud{u}{u}\}\ \right)
\left(\beta^u(\overline{U}_{i}\cap \gud{u}{u})\right)^2.$$ 
Then $M^2<\infty$, since $\Delta_u$ is continuous in $t$ and
$\overline{U}_{i}\cap\gud{u}{u}$ is compact. Then we have that 
	
\begin{eqnarray}
||f-f_{j}||_{u}&=& \left[ (f-f_{j},f-f_{j})_u  \right]^{-1/2} \nonumber \\
			&=&  \left[ \int_{\gudu}{\left((f-f_j)^{*}\ast(f-f_j)\right)(t)
		  		 \: \Delta_{u}(t)^{-1/2}\:d\beta^u(t)}  \right]^{1/2}  \nonumber\\
		  	&=& \left[\int_{\gudu}
			{\int_{\gudu}{(f-f_j)^{*}(s)
			(f-f_j)(s^{-1}t)\:d\beta^u(s)} \:\Delta_{u}(t)^{-1/2}
			\:d\beta^u(t)}\right]^{1/2}  \nonumber \\
			&\leq& \left[||f-f_{j}||_{\infty}^{2} 
			 \sup_{t\in\overline{U}_{i}\cap\gud{u}{u}} \{ \Delta_{u}(t)^{-1/2}\}
			 \:\beta^u(\overline{U}_{i}\cap \gud{u}{u})^2 \right]^{1/2} \nonumber \\
			&=& ||f-f_{j}||_{\infty}M. \label{eq:Case2Dense}
\end{eqnarray}

Since $f_{j}\to f$ uniformly and  
$\supp(f_{j})\subset\overline{U}_{i}$ for every $j$, it follows that 
there is a $j_{0}\in \mathbb{N}$ such that if $j>j_{0}$, then
$$||f-f_{j}||_{\infty}<\frac{\epsilon}{2M}.$$
Fix any $j>j_{0}$.	Then it follows from  
(\ref{estimate1}), (\ref{eq:Case1Dense}) and (\ref{eq:Case2Dense}) that 

\begin{eqnarray*}
||h-f_{j}||_{u}  
 &\leq& ||h-f||_{u} + ||f-f_{j}||_{u} \\
&\leq& ||h-f||_{u} + ||f-f_{j}||_{\infty}M \\
&<&	\epsilon.
\end{eqnarray*}
Thus every neighborhood of $h$ contains some $f_{j} \in \{f_{i_{k}}\}_{i,k}$, 
which shows that $\{f_{i_{k}}\}_{i,k}$ is dense in $\fH{u}$. 
\end{proof}
		
Next we show that 	$\bhb$ is a Borel Hilbert bundle.	
	
\begin{prop} \label{bhbprop}
There is a sequence $\{f_{i}\}$ of functions that is dense in 
$C_{c}(G)$ in the inductive limit topology. For every 
$i$ and every $u\in\gus$ put
$$\hat{g}_{i}(u):=f_{i}|_{\gud{}{u}}.$$
Then $\bhb$ is a Borel Hilbert bundle with fundamental sequence
$\{g_{i}\}$ given by 
$$  g_{i}(u):=(u,\hat{g}_{i}(u)). $$ 
\end{prop}
	
\begin{proof}
By Lemma \ref{denseinductivelim} there is a sequence $\{f_{i}\}\subset C_{c}(G)$
which is dense in $C_{c}(G)$ in the inductive limit topology. 

We continue our convention of writing $g_{i}(u)$ to mean $\hat{g}_{i}(u)$
if no confusion is possible.
We show that the conditions of Proposition \ref{bhbthm} are satisfied for 
$\bhb$ and $\{g_{i}\}$. 
Because $G$	is Hausdorff, $\gus$ is closed in $G$. Thus $\gus$ is 
also second-countable, locally compact and Hausdorff. 
Hence $\gus$ is Polish (\cite[Lemma 6.5]{CrossedProd}), and 
thus an analytic Borel space.

Next we show that the sequence of sections $\{g_{i}\}$
satisfy conditions (b2) and (b3) of Definition \ref{dfnbhb}. 
Fix $m,n\in \mathbb{N}$. Then
$$ u\mapsto ( g_{n}(u)\mid g_{m}(u))_{u}= 
 \int_{\gudu}{(f_{m}^{*}*f_{n})(t)\:\Delta_{u}(t)^{-\frac{1}{2}}
 \: d\beta^{u}(t)}, $$
which is Borel by Proposition \ref{Borelprop}. Thus (b2) is satisfied.
			
Let $\rho:\bhb\to \gus$ be the projection onto 
the first coordinate. Let $\tilde{g}_{i}$ be as in Definition \ref{dfnbhb} (b1). We show  that $\{\tilde{g_{i}}\}\cup \{\rho\}$
separate the points of $\bhb$. Suppose that $\{\tilde{g_{i}}\}\cup \{\rho\}$ 
do not separate points. Then there exist distinct points
$(u,h)$ and $(v,k)$ in $\bhb$ such that, for every 
$\phi\in\{\tilde{g_{i}}\}\cup \{\rho\}$, we have $\phi(u,h)=\phi(v,k)$.
First notice that if $\phi=\rho$, then $\rho(u,h)=\rho(v,k)$ implies 
$u=v$. Thus $k\in\fH{u}$. Then, besides $\rho$, we have 
$\tilde{g_{i}}(u,h)=\tilde{g_{i}}(u,k)$ for every $i\in \mathbb{N}$.
That is, 	
$$( g_i(u) \mid h )= ( g_i(u) \mid k ), $$
or 
$$ 0=( g_i(u) \mid h-k ) = (f_{i}|_{\gud{}{u}} \mid h-k ).$$ 
Thus $h-k$ is in the
orthogonal compliment $\{f_{i}|_{\gud{}{u}} : i\in \mathbb{N}\}^{\bot}$ in $\fH{u}$.
By Lemma \ref{denseinductivelim} the set  $\{f_{i}|_{\gud{}{u}}\}$ 
is dense in $\fH{u}$. Thus 
$\{f_{i}|_{\gud{}{u}}: i\in \mathbb{N}\}^{\bot}=\{0\}$, 
implying $h=k$. So $(u,h)=(v,k)$, which 
contradicts the assumption that these points are distinct.
So condition (b3), and hence all the conditions of Proposition \ref{bhbthm} are satisfied, showing that $\bhb$ is a Borel Hilbert bundle with fundamental sequence $\{g_{i}\}$.
\end{proof}

Next we form an $L^2$-space with sections of
$\ca{B}(\bhb)$. To form this  $L^2$-space we use a quasi-invariant measure (see for example Definition 3.1 in \cite{Ren80}).
Since we work in second-countable, locally compact Hausdorff spaces, all Borel measures are $\sigma$-finite. The class of measures equivalent to any such $\sigma$-finite measure contains a finite measure. For groupoids the notions of quasi-invariance and ergodicity depend on the measure class. So we may assume without loss of generality that our quasi-invariant measure is a probability measure.  
Also, if $G$ is second-countable, locally compact and  Hausdorff, then 
there always exists a quasi-invariant measure on $\gus$,
\cite[Proposition 3.6]{Ren80}. 
	
Let $\mu$ be a quasi-invariant measure on $\gus$ and let
$$\mathscr{L}^2(\bhb,\mu):=\{f\in \ca{B}(\bhb): 
u\mapsto || f(u)||^{2}_{u} \text{ is $\mu$-integrable}\}.$$
Let $L^2(\bhb,\mu)$ be the vector space formed by taking the 
quotient of  $\mathscr{L}^2(\bhb,\mu)$, where sections agreeing $\mu$-a.e. 
are equivalent. As is common in the literature, we use the same symbol $f$ for the class of sections in $L^2(\bhb,\mu)$ to which $f$ belongs. 
Let $f,g\in L^2(\bhb,\mu)$. The functions $u\mapsto || f(u)||^{2}_{u}$ and 
$u\mapsto || g(u)||^{2}_{u}$ belong to $L^2(\gus,\mu)$. 
By H{\"o}lder's inequality  $u\mapsto || f(u)||_u\:|| g(u)||_u$ is in $L^1(\gus,\mu)$. 	
Then, by the Cauchy-Schwarz inequality 
 \begin{eqnarray*}
		\left| \int_{\gus}{( f(u)\mid g(u))_{u}d\mu(u)} \right|	 
		&\leq&  \int_{\gus}{\big{|}( f(u)\mid g(u))_{u}\big{|}\: d\mu(u)}   \\		
		&\leq&  \int_{\gus}{|| f(u)||_{u} || g(u)||_{u}\:  d\mu(u)} \\ 
		&<& \infty.
	 \end{eqnarray*}
	Thus $u\mapsto ( f(u)\mid g(u))_{u}$ is $\mu$-integrable, and so
	$$ (f\mid g) :=\int_{\gus}{( f(u)\mid g(u))_{u}\:  d\mu(u)} $$
	defines an inner product on $L^2(\bhb,\mu)$. 
	With this inner product $L^2(\bhb,\mu)$ is a Hilbert space.
	The Hilbert space $L^2(\bhb,\mu)$ is what is known as the Hilbert space
	direct integral, also denoted by 
	$\int_{\gus}^{\oplus}{\fH{u}\: d\mu(u)}$ in the literature.
	
	We turn our attention to the direct integral representation, 
	which will act on $L^2(\bhb,\mu)$. 
	Fix $a\in C^*(G)$. 
For every $f\in L^2(\bhb,\mu)$ and $u\in \gus$, define 
$(L(a)f)(u):=(u,l^u(a)( f(u)))$. We write 
\begin{equation} \label{LaRep}
(L(a)f)(u)=l^u(a)( f(u)).
\end{equation}
to shorten notation.

\begin{prop} \label{Lrep}
Let $a\in C^*(G)$. Then $L(a)$ is a bounded linear operator on $L^2(\bhb,\mu)$, and
the map $a\mapsto L(a)$ defines a representation of 
$ C^*(G)$ on $L^2(\bhb,\mu)$. 
\end{prop}
	
\begin{proof}
Note that the linearity of $L(a)$ follows from the linearity of each $l^u(a)$.
To show that $L(a)$ is a bounded linear operator on $L^2(\bhb,\mu)$, 
we first show that $L(a)$ maps Borel sections to Borel sections. 
Let $k\in \ca{B}(\bhb)$. We claim that $L(a)k\in \ca{B}(\bhb)$.
To see this, let $\{g_{i}\}$ be the fundamental 
sequence given by Proposition \ref{bhbprop}. Recall, 
$g_i(u)=(u,g_i(u))$ with $g_i(u)=f_i|_{G_u}$, and where $\{f_i\}$ is 
a countable sequence in $C_{c}(G)$ which is dense in the inductive limit topology.
We must show that
\begin{equation*} 
	u \mapsto ((L(a)k)(u)\mid g_{n}(u))_{u}=
	(l^u(a)( k(u))\mid g_{n}(u))_{u}
\end{equation*}
is Borel for every $n$.
Since $\{g_i\}$ is a fundamental sequence, it is sufficient to show that
\begin{equation}  \label{borelsecmap}
	u \mapsto (l^u(a) g_{n}(u)\mid g_{m}(u))_{u} 
\end{equation}
is Borel for all $m$ and $n$,  \cite[Proposition 1]{DixVN}. 		
First we show that (\ref{borelsecmap}) is Borel for functions in the dense 
subspace $C_{c}(G)$ of $C^*(G)$, and then 
for an arbitrary $a\in C^*(G)$. Let $h\in C_{c}(G)$ 
and fix  $m$ and $n$. Then
\begin{eqnarray*}
	(l^{u}(h) g_{n}(u)\mid g_{m}(u))_{u} 
	&=& \int_{\gudu}{[f_{m}^{*}*(l^{u}(h) f_{n})](t)
	 \Delta_{u}(t)^{-\frac{1}{2}}d\beta^{u}(t)}  \\
	&=& \int_{\gudu}{[f_{m}^{*}*(h*f_{n})](t)
	\Delta_{u}(t)^{-\frac{1}{2}}d\beta^{u}(t)}  \\
	&=& \int_{\gudu}{(f^{*}_{m}*h*f_{n})(t)
	 \Delta_{u}(t)^{-\frac{1}{2}}d\beta^{u}(t)}.  \\
\end{eqnarray*}
Since  $f^{*}_{m}*h*f_{n}\in C_{c}(G)$, it follows from 
Proposition \ref{Borelprop} that
\begin{equation} \label{firstBorel}
 u \mapsto \int_{\gudu}{(f^{*}_{m}*h*f_{n})(t)\Delta_{u}(t)^{-\frac{1}{2}}d\beta^{u}(t)}
\end{equation}
is Borel.
Now, let $a\in C^{*}(G)$ be arbitrary. 
Then there is a sequence $\{h_{i}\}$ in $C_{c}(G)$ such that  $h_i\to a$
in the $C^*$-norm.
Put $\phi_{i}(u)=(l^{u}(h_{i}) g_{n}(u)\mid g_{m}(u))_{u}$ and 
$\phi(u)= (l^{u}(a) g_{n}(u)\mid g_{m}(u))_{u}$. By 
(\ref{firstBorel}), $\{\phi_{i}\}$ is 
a sequence of Borel measurable functions on $\gus$ for all $m$ and $n$. 
Since inner products and the $l^u$'s are continuous, 
$\phi_{i}(u)\to \phi(u)$ for every $u\in\gus$. 
This pointwise convergence implies that $\phi$ is Borel measurable.
Hence the map (\ref{borelsecmap}) is Borel and 
it follows that $L(a)f\in \ca{B}(\bhb)$.
		
Next we show that $L(a)$ maps $L^2(\bhb,\mu)$ into $L^2(\bhb,\mu)$. 
That is, we show that for every $k\in L^2(\bhb,\mu)$ and $a\in C^{*}(G)$, that
the map $u\mapsto ||(L(a)k)(u)||_{u}^{2}$ is $\mu$-integrable. In this case
we have  
\begin{eqnarray}
 \int_{\gus}{||(L(a)k)(u)||_{u}^{2}\:d\mu(u)} &=& 
 \int_{\gus}{|| l^{u}(a)(k(u))||_{u}^{2}\:d\mu(u)}  \nonumber\\
 &\leq & ||a||^2 \int_{\gus}{|| k(u))||_{u}^{2} \:d\mu(u)}  \nonumber \\
 &=& ||a||^2||k||_2^2  \label{Lbounded}\\
 &<&\infty.  \nonumber
\end{eqnarray}
Thus the map $u\mapsto ||(L(a)k)(u)||_{u}^{2}$ is $\mu$-integrable, 
showing that  $L(a)k\in L^2(\bhb,\mu)$.
		
The boundedness of $L(a)$
follows from (\ref{Lbounded}), since 
$$||L(a)k||_2^2 = \int_{\gus}{||l^u(a) k(u)||^{2}_{u} \,\:  d\mu(u)} 
\leq ||a||^2||k||_2^2. $$
Hence $L(a)\in B(L^2(\bhb,\mu))$ for every $a\in C^*(G)$. 

Lastly, that $a\mapsto L(a)$ is a representation of $C^*(G)$ follows from 
(\ref{LaRep}) and the fact that every $l^u$ is a representation of $C^*(G)$. 
\end{proof}
	
The representation  $L$ of (\ref{LaRep}) is a 
\textit{direct integral representation}, \cite[Definition 8.1.3]{DixC},
and is denoted  by 
$$  L:=\int_{\gus}^{\oplus}{l^{u}\: d\mu(u)}. $$
The operators $L(a), \, a\in C^*(G),$ are called \textit{decomposable operators}, and 
are denoted by 
$$  L(a):=\int_{\gus}^{\oplus}{l^{u}(a)\: d\mu(u)}.$$

We need a few last remarks on von Neumann algebras and some lemmas which we use to prove a groupoid version of Effros' lemma for transformation groups.
	
If $\fH{}$ be a Hilbert space and $\mathcal{M}$ a self-adjoint subset of $B(\fH{})$, then we denote by $\mathcal{M}^{\prime}$ the \textit{commutant} of $\mathcal{M}$. We say 
$\mathcal{M}$ is a von Neumann algebra if $\mathcal{M}=\mathcal{M}^{\prime\prime}$.
The \textit{centre} of a von Neumann algebra $\mathcal{M}$ is the abelian von Neumann algebra
$$\mathcal{Z}(\mathcal{M}):=\mathcal{M}^{\prime}\cap \mathcal{M}.$$

	\begin{lem} \label{CentreInMaximal}
		Let $\mathcal{M}$ be a von Neumann algebra and suppose that $\mathcal{N}$ is a maximal abelian von Neumann subalgebra of $\mathcal{M}^{\prime}$, 
		in the sense that $\mathcal{N}$ is not properly contained in any other
		abelian von Neumann subalgebra of $\mathcal{M}^{\prime}$. Then 
		 $\mathcal{Z}(\mathcal{M}) \subset \mathcal{N}$.
	\end{lem}
	
	 \begin{proof}
		The structure of the proof is as follows: we first show that the von Neumann 
		algebra $\overline{K}^{sot}$ generated by 
		$\mathcal{Z}(\mathcal{M})$ and $\mathcal{N}$ is 
		a von Neumann subalgebra of $\mathcal{M}^{\prime}$. Then we show that
		$\overline{K}^{sot}$ is an abelian von Neumann algebra. 
		Lastly, we show that $\overline{K}^{sot}$ 
		contains both $\mathcal{Z}(\mathcal{M})$ and $\mathcal{N}$.
		By the maximality of  $\mathcal{N}$, we then have 
		$\mathcal{Z}(\mathcal{M})\subset \mathcal{N}$.
		
		Suppose that $z\in \mathcal{Z}(\mathcal{M})$ and $n\in \mathcal{N}$. Then 
		$z\in \mathcal{M}$ and $n\in \mathcal{M}^{\prime}$. Thus 
		$zn=nz$. Hence any product formed from elements of $ \mathcal{Z}(\mathcal{M})$ and 
		$\mathcal{N}$ are of the form $zn$, with 
		$z\in \mathcal{Z}(\mathcal{M})$ and $n\in \mathcal{N}$. 
		Let $K= \Span \{zn: z\in \mathcal{Z}(\mathcal{M}), n\in \mathcal{N}\}$. Then $K$ is 
		an abelian *-subalgebra of $\mathcal{M}^{\prime}$, since $\mathcal{M}^{\prime}$
		is itself a von Neumann algebra and 
		contains both $\mathcal{Z}(\mathcal{M})$ and $\mathcal{N}$. 
		Then von Neumann's Double Commutant Theorem,
		\cite[Theorem 2.4.11]{BratteliRobinson}, implies the strong 
		operator closure $\overline{K}^{sot}$ 
		of $K$ is a von Neumann subalgebra of $\mathcal{M}^{\prime}$. Also,
		$\overline{K}^{sot}$  contains both
		$\mathcal{Z}(\mathcal{M})$ and $\mathcal{N}$.
	
We claim that $\overline{K}^{sot}$ is abelian.  First suppose that $S\in K$ and $T \in \overline{K}^{sot}$. Let $(T_{\alpha})$ be a net in $K$ converging to $T$ in the strong operator topology. The maps $T\to ST$ and $T\to TS$ are continuous in the strong operator topology (for the fixed $S$). Thus $T_{\alpha}S \to TS$. Since $K$ is an abelian *-algebra we also have that $T_{\alpha}S=ST_{\alpha} \to ST$. 
		Since the strong operator topology is a Hausdorff, it follows that 
			\begin{equation} \label{commute}
				ST=TS.
			\end{equation}
		Now let $S,T\in \overline{K}^{sot}$, and let $\{S_{\alpha}\}$ be a sequence in $K$ such that $S_{\alpha}\to S$
		in the strong operator topology. Applying Equation (\ref{commute}), we have that 
		$S_{\alpha}T=TS_{\alpha}$ for every $\alpha$. 
		Taking the limit now shows that $ST=TS$.
%
		
		We showed that $\overline{K}^{sot}$ is an abelian von Neumann subalgebra of 
		$\mathcal{M}^{\prime}$ which 
		contains both $\mathcal{Z}(\mathcal{M})$ and $\mathcal{N}$. But, $\mathcal{N}$ is 
		maximal abelian in $\mathcal{M}^{\prime}$. Thus $\overline{K}^{sot}=\mathcal{N}$ and  
		$\mathcal{Z}(\mathcal{M})\subset \mathcal{N}$.

%
	 \end{proof}

\begin{dfn} \label{projmeas}
Let $\mathscr{B}(\gus)$ be the Borel subsets of $\gus$. A 
\textit{projection-valued measure} $E$ for $L^2(\bhb,\mu)$ 
is a function from 
$\mathscr{B}(\gus)$ into the set of orthogonal projections on $L^2(\bhb,\mu)$
such that 
 \begin{enumerate}
	\item[(i)] $E(\gus)=1$  
	\item[(ii)] $E(A\cap B)=E(A)E(B)$, for $A,B\in \mathscr{B}(\gus)$, and  
	\item[(iii)] $E(\cup A_{i})=\sum{E(A_{i})}$ for   
	pairwise disjoint Borel subsets $A_{i}$. 
 \end{enumerate}
For $A\subset\gus$ let $1_{A}$ denote the characteristic function on  
$A$. Then, for every $f\in L^2(\bhb,\mu)$, 
\begin{equation} \label{eq:ProjMeas}
(E_{A}f)(u):=1_{A}(u)f(u)
\end{equation}
defines a projection-valued measure called
the \textit{canonical projection-valued measure}. Note: for a fixed $A\subset\gus$,
the projection $E_{A}$ is the decomposable operator
$E_{A}=\int_{\gus}^{\oplus}{1_{A}(u)\:I_{u}\:d\mu(u)}.$ 
\end{dfn}

Applying \cite[Corollary IV.12]{Fabec} to our specific direct integral representation
$ L=\int_{\gus}^{\oplus}{l^{u}d\mu(u)}$ of $C^*(G)$ 
gives:
			
\begin{prop} \cite[Corollary IV.12]{Fabec}  \label{projmeas}
The representations $l^{u}$ in the direct integral representation 
$L=\int_{\gus}^{\oplus}{l^{u}d\mu(u)}$ are $\mu$-almost all irreducible 
if and only if the range of the canonical projection valued measure 
for $L^2(\bhb,\mu)$ is a maximal abelian algebra of projections in $L(C^{*}(G))'$. 
\end{prop}

 A set $A\subset\gus$ is \textit{invariant} if $r(s^{-1}(A))=A$. 
 If $\mu$ is a quasi-invariant measure on $\gus$, then $\mu$ is 
 \textit{ergodic} if $\mu(A)=0$ or $\mu(\gus\backslash A)=0$ for all
 invariant sets $A\subset\gus$. If, in addition,
 $\mu$ is concentrated on an orbit, then $\mu$ is \textit{trivially ergodic}.	
	
\begin{lem} \label{analyticmeas}
	Let $A\subset \gus$ be a Borel set and $\mu$ a 
	Borel measure on $\gus$. Then $r(s^{-1}(A))$ is measurable.
\end{lem}		
\begin{proof}
The continuity of $r$ and $s$ imply that they are Borel measurable. Hence $s^{-1}(A)$ is 
a Borel subset of $G$ and thus analytic. Then $r(s^{-1}(A))$ is an analytic set 
in $\gus$ by \cite[Theorem 3.3.4 and Corollary 1]{ArvInvitation}.
Hence $$r(s^{-1}(A))$$ is $\mu$-measurable by \cite[Theorem 3.2.4]{ArvInvitation}.  
\end{proof}

Suppose that $\pi$ is a representation of
a $C^*$-algebra $\ca{A}$ on a Hilbert space $\fH{}$.
Then $\pi$ is a \textit{factor representation} if the centre 
of the von Neumann algebra $\pi(\mathcal{A})^{\prime \prime}$ 
consists of scalar multiples
of the identity operator on $\fH{}$. 
If the 	the von Neumann algebra $\pi(\mathcal{A})''$ is a type I von Neumann
algebra then $\pi$ is a \textit{type I representation}.
The $C^*$-algebra $\mathcal{A}$ is type I  if every 
representation of $\mathcal{A}$ is type I.

We now  adapt Effros' proof of 
\cite[Lemmma 4.2]{Eff65} from transformation groups to groupoids.
We only need one direction 
of Effros' `if and only if' statement, that is, type I implies trivially ergodic.

\begin{prop} \label{factorrep}
	Let $\mu$ be an ergodic Borel measure on $\gus$. Then 
\begin{enumerate}
	 \item[(i)] the direct integral representation 
	   $ L=\int_{\gus}^{\oplus}{l^{u}d\mu(u)}$ is a factor representation
	   of $C^*(G)$, and
	 \item[(ii)] if $L$ is a type I factor representation, then
	   $\mu$ is trivially ergodic.
\end{enumerate}
\end{prop}
	
We give a brief overview of the structure of the proof of (i) in an attempt to 
make the proof easier to read. We prove the contrapositive of (i), 
and split the proof into three main 
parts. \textit{Existence of a direct integral projection:} We show that if $L$ is not a 
factor representation, then there is a projection $P$ and a Borel set 
$B\in\mathscr{B}(\gus)$ such that $P=\int_{\gus}{1_{B}(u)I_{u}d\mu(u)}$ with 
$P\neq I$ and $P\neq 0$. 
\textit{Convergence of integrands off a null set $N$:} We show that there is 
sequence $\{a_{n}\}\subset C^{*}(G)$ and a null set $N\in\mathscr{B}(\gus)$
such that $l^{u}(a_{n_{k}}) \to 1_{B}(u)I_{u}$ strongly 
for every $u\in \gus\backslash N$.
\textit{An invariant set neither null nor conull:} Lastly, we show that 
$C:=r(s^{-1}(B\backslash N))$ is invariant, and is neither 
null nor conull under $\mu$. That is, $\mu$ is not ergodic.
\begin{proof}
(i) \textit{Existence of a direct integral projection:} 
 Suppose that $L$ is not a  factor representation of $C^*(G)$.  Then the centre 
 $L(C^{*}(G))^{\prime}\cap L(C^{*}(G))^{\prime \prime}$ has
 elements other than multiples of the identity.
 Since von Neumann algebras are generated by their projections 
 (\cite[Theorem 4.1.11]{Murphy}) and the centre is an 
 abelian von Neumann algebra, there exists a projection  
$P\in L(C^{*}(G))^{\prime}\cap L(C^{*}(G))^{\prime \prime}$ with $P\neq 0$ and 
$P\neq I$ ($I$ being the identity operator in $B(L^2(\bhb,\mu))$).
Let $E$ denote the canonical projection-valued measure from $\mathscr{B}(\gus)$
into the set of orthogonal projections in 
$B(L^2(\bhb,\mu))$. 
Since every $l^{u}$ is irreducible, it follows 
from Proposition \ref{projmeas} that the range $E(\mathscr{B}(\gus))$
of the canonical projection-valued 	measure 
is a maximal abelian
algebra in $L(C^{*}(G))^{\prime}$.
Thus the set of projections contained
in the image of the canonical projection-valued measure generates 
a maximal abelian von Neumann subalgebra 
$E(\mathscr{B}(\gus))^{\prime \prime}$ of $L(C^{*}(G))^{\prime}$. 
Then $L(C^{*}(G))^{\prime}\cap L(C^{*}(G))^{\prime \prime}\subset
E(\mathscr{B}(\gus))^{\prime \prime}$, by  Lemma \ref{CentreInMaximal}. Hence
$P\in E(\mathscr{B}(\gus))^{\prime \prime}$. 
Because $E(\mathscr{B}(\gus))^{\prime \prime}$ is generated by its 
projections 
$E(\mathscr{B}(\gus))$, it follows 
that $P\in E(\mathscr{B}(\gus))$. Hence there is a
Borel set $B$ in $\mathscr{B}(\gus)$ such that $0\neq \mu(B) \neq \mu(\gus)$
and 
$P=E_{B}=\int_{\gus}{1_{B}(u)I_{u}d\mu(u)}$ 
(where $1_B$ is the indicator function of $B$).
	
\textit{Convergence of integrands off a null set $N$:}
By \cite[Lemmma 3.1]{Eff63} $L$ is a non-degenerate representation because 
each $l^{u}$ is non-degenerate.
By von Neumann's double commutant theorem, 
\cite[Theorem 2.4.11]{BratteliRobinson}, 
$L(C^{*}(G))$ is dense in $L(C^{*}(G))^{\prime \prime}$ in the strong operator
topology. By Kaplansky's density theorem the unit ball of 
$L(C^{*}(G))$ is strongly dense in the unit ball of $L(C^{*}(G))^{\prime \prime}$,
\cite[Theorem 5.3.5]{KadisonRingrose1}. With Kaplansky's density theorem and 
because the unit ball 
of $L(C^{*}(G))^{\prime \prime}$ is metrizable				
in the strong operator topology, 
\cite[Proposition 1]{DixVN}, we may replace nets with sequences in the 
strong closure of the unit ball of $L(C^{*}(G))$. Thus, since $||E_{B}||=1$, there is 
a sequence $\{a_{n}\}\subset C^{*}(G)$ such that $L(a_{n})$ is in the unit ball of 
$L(C^*(G))^{\prime \prime}$, and 
$$ L(a_{n})=\int_{\gus}^{\oplus}{l^{u}(a_{n})\: d\mu(u)}
 \to E_{B}=\int_{\gus}{1_{B}(u)I_{u}\: d\mu(u)}$$
in the strong operator topology.
Now, due to the strong convergence of the sequence of direct integrals above, 
 \cite[Proposition 4 of Part II, Chapter 2]{DixVN} tells us there is a 
subsequence $a_{n_{k}}$ such that for $\mu$-a.e. $u\in\gus$
\begin{equation}\label{repsconv} 
l^{u}(a_{n_{k}}) \to 1_{B}(u)I_{u}    
\end{equation}
strongly. That is, there is a 
Borel set $N\subset \gus$ such that $\mu(N)=0$ and 
$l^{u}(a_{n_{k}}) \to 1_{B}(u)I_{u}$ strongly for every $u\in \gus\backslash N$.
			
\textit{An invariant set neither null nor conull:}
Let $C:=r(s^{-1}(B\backslash N))$. Then $C$ is an invariant subset of $\gus$. 
Moreover $C$ is measurable by Lemma \ref{analyticmeas}. It will suffice to show  
that $\mu(C)=\mu(B)$, because  then $\mu(C)\neq 0$ and 
$\mu(\gus \backslash C)=\mu(\gus) -\mu(C)= \mu(\gus) -\mu(B)\neq 0$, which shows that 
$\mu$ is not ergodic. Note, since $B\backslash N \subset C$, it follows that
$\mu(B\backslash N)\leq \mu(C)$. Then, since 
$\mu(N)=0$, we get 
\begin{equation}\label{eq:MeasEstimate}
\mu(B\backslash N)=\mu(B)-\mu(N)=\mu(B).
\end{equation}
Thus $\mu(B)\leq \mu(C).$ 
Similarly, to get the 
reverse inequality we show that $C\backslash N \subset B$. Suppose that
$w\in C\backslash N$.
Since $C=\{u\in\gus:u\in[v] \text{ and } v\in B\backslash N\}$, there is
a $v\in B\backslash N$ such that 
$v$ is equivalent to $w$.  Lemma 5.1 of \cite{Cla07} shows that the 
map $[u] \mapsto [l^u]$ from the orbit space $\gus/G$ into
the spectrum $C^*(G)^{\wedge}$ is well-defined. Lemma 5.5 of  \cite{Cla07}
shows that this map  $[u] \mapsto [l^u]$ is injective. Hence
$v$ is equivalent to $w$ if 
and only if $l^v$ is unitarily equivalent to $l^w$. So since $v\sim w$ there is 
a unitary operator $U:\fH{v} \to \fH{w}$ such that $l^{w}(a)=Ul^{v}(a)U^{*}$,
 for every 
$a\in C^{*}(G)$. Since $w\notin N$, we apply (\ref{repsconv}) to $l^w$, that is, 
$$l^{w}(a_{n_{k}}) \to 1_{B}(w)I_{w}.$$ 
Since $v\in B\backslash N$, it 
follows that $1_B(v)=1$, and we can also apply (\ref{repsconv}) to $l^v$. Then 
$$l^{w}(a_{n_{k}})=Ul^{v}(a_{n_{k}})U^{*} \to U1_{B}(v)I_{v}U^{*}=UI_{v}U^{*}=I_{w}.$$ 
Limits are unique in the strong operator topology. Thus $1_{B}(w)I_{w}=I_{w}$,
which implies that $1_{B}(w)=1$. Hence $w\in B$. That is, $C\backslash N \subset B$.
Now a similar computation to (\ref{eq:MeasEstimate}) shows that 
$\mu(C)\leq \mu(B).$ Hence $\mu(C)=\mu(B)$, proving (i).
	
(ii) Suppose that $L$ is a factor representation of type I. Then by
 \cite[Theorem 2.7]{MackeyIndRep2},
almost all $l^{u}$ are unitarily equivalent.
That is, there is a conull set $A\subset \gus$ 
	such that $u\sim v$ for all $u,v\in A$. Suppose that $u\in A$. Then 
	$A\subset \{w\in\gus:w\sim u\}$ and $\mu(\{w\in\gus:w\sim u\})\neq0$. 
	On the other hand, $\mu(\{v\in\gus:v\nsim u\})=0$. Thus 
\begin{eqnarray*}
\mu(\gus)&=&  \mu(\{w\in\gus:w\sim u\} \cup (\gus\backslash 
				\{w\in\gus:w\sim u\}))    \\
		 &=&  \mu(\{w\in\gus:w\sim u\}) + 
		 	\mu(\{v\in\gus:v\nsim u\})  \\
		 &=&  \mu(\{w\in\gus:w\sim u\}).
\end{eqnarray*}
Hence $\mu$ is concentrated on an orbit, and is thus trivially ergodic.			 
\end{proof}

\section{CHARACTERIZING GCR GROUPOID $C^*$-ALGEBRAS} \label{main}
After one last lemma we prove Theorem \ref{mainGCRthm} which says that if $C^*(G)$ is type I (or equivalently GCR) then $\gus/G$ is $T_{0}$.
	
	\begin{lem} \label{Fsigma}
	Let $G$ be a second-countable, locally compact and Hausdorff \\ groupoid.
	Let $R: G\to \gus\times\gus$, defined by $R(\gamma):=(r(\gamma),s(\gamma))$, be the 
	equivalence relation induced on $\gus$. Then $R(G)$ is an
	$F_{\sigma}$ subset in $\gus\times \gus$.
	\end{lem}
	
		\begin{proof}
		  Since $G$ is second-countable and locally compact we can express $G$ 
		  in form $G=\cup_{i=1}^{\infty}U_{i}$, where each 
			$U_{i}$ is a  neighborhood with 
		  compact closure. Since  
			the range and source maps are continuous and $\gus$ is Hausdorff, it 
			follows that
		  $r(\overline{U_{i}})$ and $s(\overline{U_{i}})$ 
		  are compact for every $i$, and thus closed in $\gus\times \gus$. 
			So $R(G)=R(\cup_{i=1}^{\infty}\overline{U}_{i})=
			\cup_{i=1}^{\infty}R(\overline{U}_{i})$
			is an $F_{\sigma}$ set in $\gus \times \gus$.
	 \end{proof}

	\begin{thm} \label{mainGCRthm}
		Let $G$ be a second-countable locally compact and Hausdorff \\ groupoid
		with a Haar system. If $C^*(G)$ is type I 
		then $\gus/G$ is $T_{0}$.
	\end{thm}
	
	 \begin{proof}
		We prove the contrapositive. Suppose that $\gus/G$ is not $T_{0}$. 
		By Lemma \ref{Fsigma} the hypotheses of 
		\cite[Theorem 2.1]{Ram90} are satisfied. So there exists
		a non-trivial ergodic measure $\mu$ on $\gus$.  
		By Proposition \ref{factorrep} the direct integral representation 
		$L=\int{l^ud\mu}$ is a non-type I factor representation. Hence 
		$C^*(G)$ is not type I, which concludes the proof.		
	 \end{proof}

Combining Theorem \ref{mainGCRthm} and Clark's Theorem 7.1 in \cite{Cla07} we can formulate a refined characterization 
of GCR groupoids $C^*$-algebras without amenability: 	
	
\begin{thm} \label{mainGCRthm2}
		Let $G$ be a second-countable, locally compact and Hausdorff \\ groupoid
		with a Haar system. Then $C^*(G)$ is 
		GCR if and only if the stability subgroups of $G$ are GCR and $\gus/G$ is $T_{0}$.
\end{thm}

\section*{ACKNOWLEDGEMENTS}
This paper forms part of work done in the author's doctoral thesis. I thank my advisors 
Astrid an Huef and Lisa Orloff Clark for their guidance, support and endless patience   throughout my studies.

\bibliographystyle{acm}

\end{document}